\documentclass[a4paper,fleqn]{cas-sc}
\usepackage{graphicx}
\usepackage{lineno,hyperref}
\usepackage{tabularx} 
\usepackage{booktabs}
\usepackage{amsmath,amssymb,amsfonts}
\usepackage{stackengine}
\usepackage{amssymb}
\usepackage{eqnarray} 
\usepackage{mathalfa} 
\usepackage[dvips]{epsfig}   
\usepackage[ruled,vlined,linesnumbered]{algorithm2e}
\usepackage{psfrag}
\usepackage{amsmath}
\usepackage{stfloats}
\usepackage{amssymb}
\usepackage{dsfont}
\usepackage{xcolor}
\usepackage{bbm}
\usepackage{float}
\usepackage[prependcaption,colorinlistoftodos]{todonotes}
\usepackage{hyperref}
\usepackage{mathrsfs}
\usepackage{lscape}
\usepackage{longtable}
\usepackage{rotating}
\usepackage{multirow}
\usepackage{amsthm}
\usepackage{color}
\usepackage{xcolor}
\usepackage{url}
\usepackage{subfigure}
\usepackage{hyperref}
\usepackage{tikz}
\usetikzlibrary{shapes.geometric, arrows}
\hypersetup{
    colorlinks=true, 
    linktoc=all,     
    linkcolor=blue,  
}

\usepackage{fancyhdr}
\usepackage{chngcntr}

\usepackage{totcount}
\usepackage{atbegshi}

\usepackage[labelfont=bf]{caption} 
\usepackage{caption} 

\usepackage{enumerate}
\usepackage{comment}
\usepackage[utf8]{inputenc}

\usepackage[prependcaption,colorinlistoftodos]{todonotes}

\newtheorem{Problem}{Problem}

\newcommand{\R}{\mathbb{R}}
\newcommand{\E}{\mathbb{E}}

\newcommand{\bv}[1]{\mathbf{#1}}

\newcommand{\sD}{\mathcal{D}}

\newcommand{\sX}{\mathcal{X}}
\newcommand{\sU}{\mathcal{U}}

\newcommand{\DKL}[2]{{D}_{\text{KL}}\left(#1\mid \mid #2 \right)}
\newcommand{\Dr}[2]{{D}_{\text{r}}\left(#1\mid \mid #2 \right)}
\newcommand{\nonnegative}[1]{\left({#1}\right)_{+}}

\newcommand{\plant}[2]{p_{{#1}} \left(\bv{x}_{{#1}}\mid \bv{x}_{{#2}}, \bv{u}_{{#1}} \right)}
\newcommand{\shortplant}[1]{p_{{#1}}^{(x)}}
\newcommand{\policy}[2]{p_{{#1}}\left(\bv{u}_{{#1}}\mid \bv{x}_{{#2}} \right)}
\newcommand{\shortpolicy}[2]{p_{{#1}\mid{#2}}^{(u)}}

\newcommand{\optimalpolicy}[2]{p^{\star}_{{#1}} \left(\bv{u}_{{#1}}\mid \bv{x}_{{#2}} \right)}

\newcommand{\optimalpolicygen}[3]{p^{\star,{#3}}_{{#1}} \left(\bv{u}_{{#1}}\mid \bv{x}_{{#2}} \right)}

\newcommand{\refplant}[2]{q_{{#1}} \left(\bv{x}_{{#1}}\mid \bv{x}_{{#2}}, \bv{u}_{{#1}} \right)}

\newcommand{\refpolicy}[2]{q_{{#1}}\left(\bv{u}_{{#1}}\mid \bv{x}_{{#2}} \right)}

\newcommand{\shortplantstat}[2]{p^{(x)}}

\newcommand{\shortpolicystat}[2]{p^{(u)}}

\newcommand{\shortoptimalpolicystat}[2]{p^{(u),\star}}

\newcommand{\shortrefplantstat}[2]{q^{(x)}}

\newcommand{\shortrefpolicystat}[2]{q^{(u)}}

\newcommand{\jointxu}[2]{p_{{#1}}\left(\bv{x}_{{#1}},\bv{u}_{{#1}}\mid \bv{x}_{{#2}} \right)}

\newcommand{\refjointxu}[2]{q_{{#1}}\left(\bv{x}_{{#1}},\bv{u}_{{#1}}\mid \bv{x}_{{#2}} \right)}

\newcommand{\shortjointxustat}[2]{p}

\newcommand{\shortrefjointxustat}[2]{q}

\newcommand{\TsD}{TsD}

\usepackage[numbers]{natbib}

\newtheorem{theorem}{Theorem}[section]

\newtheorem{lemma}[theorem]{Lemma}
\newtheorem{definition}[theorem]{Definition}

\newtheorem{remark}[theorem]{Remark}

\begin{document}
\let\WriteBookmarks\relax
\def\floatpagepagefraction{1}
\def\textpagefraction{.001}

\shorttitle{Tsallis Probabilistic Design}

\shortauthors{Kungurtsev, Russo}

\title [mode = title]{Towards Tsallis Fully Probabilistic Design}                      



%
\author[1]{Vyacheslav Kungurtsev}[orcid = 0000-0003-2229-8824]

\cormark[1]


\ead{kunguvya@fel.cvut.cz}



\affiliation[1]{organization={Department of Computer Science, Czech Technical University},
    city={Prague},
    country={Czech Republic}}

\author[2]{Giovanni Russo}[orcid = 0000-0001-5001-3027]

\ead{giovarusso@unisa.it}


\affiliation[2]{organization={Department of Computer and Electrical Engineering \& Applied Mathematics, University of Salerno},
    city={Salerno},
    country={Italy}}

\cortext[cor1]{Corresponding author}



\begin{abstract}
Fully Probabilistic design (FPD) is a powerful framework offering an elegant and unifying account of stochastic control, learning and decision-making. Here we introduce a generalized FPD framework, which we term as Tsallis FPD. Tsallis FPD uses Tsallis divergence in place of the Kullback-Leibler divergence that defines the standard FPD cost term. Tsallis divergence is a natural generalization of the KL divergence, rooted in non-extensive statistical mechanics and providing flexibility towards modeling stochastic processes with non-Gaussian tail behavior. After formulating Tsallis FPD, we present a double iteration scheme that performs a sequence of backwards inductions, rather than a single pass down the stages that constitutes the proven approach for classical FPD. 
\end{abstract}




\maketitle

\section{Introduction}

Fully probabilistic design (FPD) offers an elegant and unifying account of stochastic control, learning and decision making. Introduced in~\cite{karny1996towards}, FPD defines a sequential Bayesian decision-making problem that involves finding optimal policies by minimizing an information-theoretic objective: the Kullback-Leibler (KL) divergence~\cite{SK_RL:51} between two probability densities. These densities capture the (agent-environment) closed-loop behavior minimal in divergence to some ideal reference behavior. 

Since its introduction, FPD has seen substantial development over the years; see, e.g.,~\cite{MK-TG:06,karny2012axiomatisation,RH-15,AQ-MK-TG:16,MK:20,gagliardi2022probabilistic,MK-TS:23,MK-SM:25}. Beyond this rich body of work, FPD is closely related to a broad range of decision-making frameworks, such as knowledge transfer~\cite{foley2017fully}, information fusion~\cite{azizi2016hierarchical}, KL control~\cite{ET:09,PG_MR_RW:14}, control-as-inference~\cite{HK-VC-MO:12,MahmoudiFilabadi2026} and maximum entropy reinforcement learning~\cite{eysenbach2022maximum}; see also~\cite{EG-GR:22} for a survey and the closely related~\cite{bhole2025unifyingentropyregularizationoptimal}. Moreover, the information theoretical cost functional in FPD is also closely related to policy computation under the free energy principle in neuroscience, see, e.g.,~\cite{KF:09,AS-HJ-KF-GR:25}. The choice of the KL divergence in FPD yields several remarkable properties that enable deriving closed-form solutions for the optimal policy. The KL is however a particular instance within a broader class of information measures. The Tsallis divergence constitutes a natural generalization of the KL divergence, rooted in non-extensive statistical mechanics, with the associated Tsallis entropy extending Shannon entropy through a non-additive structure~\cite{CT:88,CT-GMM-YS:05}. As noted in~\cite{lee2019tsallis}, by tuning an entropic index, Tsallis divergence can generate various types of divergences and recovers the KL divergence in a special case. This flexibility can be useful to, e.g., effectively model long range and fat tail dependency~\cite{lee2019tsallis}.  In fact, as also noted in~\cite{wang2021variational}, an increasing body of theoretical and experimental literature has shown that a wide range of complex natural, artificial, and social systems can be well described by its flexible stochastic structure. Consequently, Tsallis divergence is increasingly leveraged to regularize reward/cost functions in learning, optimal transport and control~\cite{lee2019tsallis,LZ-ZC-MS-MW:23,nielsen2011r,wang2021variational,10.5555/3298483.3298583} to, e.g., improve exploration and a flexible general framework for policies associated with different dynamic structure.

Motivated by these observations, we introduce a generalization of FPD, Tsallis FPD (T-FPD), where Tsallis divergence is used in place of the KL divergence in the FPD formulation. In FPD, computing the optimal policy involves a backward recursion that exploits additivity of the KL divergence. This approach cannot be used in T-FPD since Tsallis divergence is not additive. To address this challenge, we  formulate a double iteration scheme. 



\noindent{\bf Contributions.} To the best of our knowledge, there is no work in the literature extending FPD to Tsallis divergence. This work fills this gap. Our key technical contributions can be summarized as follows. We introduce Tsallis FPD formally as a probabilistic decision problem. However, finding the optimal solution requires developing a distinct mathematical approach. Deriving a constructive method to find the optimal solution, is a key contribution of our work. Specifically, we obtain a solution of the T-FPD by deriving the backwards induction to a multistage FPD problem by a Jacobi iteration type decomposition of fixed and free quantities for the solution density in the optimization expression. This defines a double loop fixed point-type iteration, whose fixed points correspond to solutions to the original problem. 

The rest of the paper is organized as follows. After introducing the mathematical background in Section~\ref{sec:prelim}, we formulate Tsallis Fully Probabilistic Design and the state our main results  (Section~\ref{sec:mainresults}). The derivations of the expressions for the induction are given in Section~\ref{sec:proofs}. Concluding remarks are given in Section~\ref{sec:conc}. 

\section{Mathematical Preliminaries}\label{sec:prelim}
Sets are in {\em calligraphic} and vectors in {\bf bold}.   A random variable is denoted by $\bv{V}$ and its realization is $\bv{v}$. We denote the \textit{probability mass function} (pmf, for discrete variables) or \textit{probability density function} (pdf, for continuous variables) of $\bv{V}$  by $p(\bv{v})$ and we let $\sD$ be the convex subset of pdfs/pmfs, that is for an implied domain $\R^d$, $\sD:=\{p\in L^1(\R^d),\,\|p\|_{L^1(\R^d)}=1,\, p(\bv{v})\ge 0 \,\forall \bv{v}\in\R^d\}$. Whenever we take  integrals we always assume that they exist.  The  expectation of a function $\bv{h}(\cdot)$ of $\bv{V}$ is denoted $\E_{{p}}[\mathbf{h}(\bv{V})]:=\int_{\bv{v}}\bv{h}(\bv{v})p(\mathbf{v})$, where the integral is over the support of $p(\bv{v})$; whenever it is clear from the context, we  omit the subscript in the integral.
The {joint} pdf of $\bv{V}_1$ and $\bv{V}_2$ is denoted by  $p(\bv{v}_1,\bv{v}_2)$ and the {conditional} pmf/pdf of $\bv{V}_1$ with respect to (w.r.t.) $\bv{V}_2$ is $p\left( \bv{v}_1\mid   \bv{v}_2 \right)$. Countable sets are denoted by $\lbrace w_k \rbrace_{k_1:k_n}$, where $w_k$ is the generic set element, $k_1$ ($k_n$) is the index of the first (last) element and  $k_1:k_n$ is the set of consecutive integers between (including) $k_1$ and $k_n$.
Also,  functionals are denoted by capital calligraphic characters with arguments within curly brackets. Standard finite indexing notation will be used in the form $[N]:=\{1,2,3,\cdots,N\}$.  


\subsection{Tsallis Divergence: definition and basic properties}


To formulate T-FPD, we make use of Tsallis Divergence (\TsD). We use the definition from \cite{wang2021variational}, which is stated for continuous random variables. We first define a parametric deformation of the logarithm and exponential:

\begin{definition}
The deformed logarithm $\log_r:\R^+ \to \R$ and deformed exponential $\exp_r:\R\to\R$ are defined as:
\begin{align}
\log_r(v) & := \frac{v^{r-1}-1}{r-1}\\
\exp_r(v) & := \nonnegative{1+(r-1)v}^{\frac{1}{r-1}},
\end{align}
with $r>0$ and $\nonnegative{x}:=\max\left\{0,x\right\}$.
\end{definition}
With this definition, we can then give the following:
\begin{definition}[Tsallis Divergence and Total Conditional Tsallis Divergence]
Let $p(\bv{v})$, $q(\bv{v})$ be pdfs. The Tsallis Divergence (\TsD) of $p(\bv{v})$ w.r.t. $q(\bv{v})$ is 
$$
\Dr{p(\bv{v})}{q(\bv{v})} := \E_p\left[\log_r\frac{p(\bv{v})}{q(\bv{v})}\right].
$$
The {Total Conditional Tsallis Divergence} of $p(\bv{v} \vert \bv{z})$, with respect to $q(\bv{v})$ is
$$
\E_{p(\bv{z})}\left[\Dr{p(\bv{v}\vert \bv{z})}{q(\bv{v})}\right] := \E_{p(\bv{z})}\left[\E_{p(\bv{v}\vert\bv{z})}\left[\log_r\frac{p(\bv{v}\vert \bv{z})}{q(\bv{v})}\right]\right].
$$
\end{definition}
\begin{remark}
In the above definition, $r$ is often termed as entropic index. As noted in \cite{tsallis1988possible}, as $r\to 1$, TsD becomes the Kullback-Leibler (KL) Divergence. As for the Kullback-Leibler Divergence, TsD is finite only if $p(\bv{v})$ is absolutely continuous w.r.t. $q(\bv{v})$.
\end{remark}
An important distinction between KL divergence and Tsallis divergence is that \TsD~is sub-additive. This crucial difference is formalized with the following result adapted from~\cite{furuichi2012some}:
\begin{lemma}\label{lem:non-additivity}
Let $p(\bv{v},\bv{z})$, $q(\bv{v},\bv{z})$ be two joint pdfs such that $q(\bv{v},\bv{z})=q_v(\bv{v})q_z(\bv{z})$ and $p(\bv{v},\bv{z})=p_{\vert}(\bv{v}\vert\bv{z})p_z(\bv{z})$. Then:
    \begin{equation}\label{eq:non-additive}
\begin{array}{l}
\Dr{p(\bv{v},\bv{z})}{q(\bv{v},\bv{z})}  = \E_{p(\bv{z})}\left[\Dr{p(\bv{v}\vert \bv{z})}{q(\bv{v})}\right] +\Dr{p(\bv{z})}{q(\bv{z})} +(1-r)\E_{p(\bv{z})}\left[\Dr{p(\bv{v}\vert \bv{z})}{q(\bv{v})}\right]\Dr{p(\bv{z})}{q(\bv{z})}.
        \end{array}
\end{equation}
\end{lemma}
\begin{remark}
From Lemma~\ref{lem:non-additivity}, as $r\to 1$, the last expression in~\eqref{eq:non-additive} vanishes. This means that, only in this special case, Lemma~\ref{lem:non-additivity} yields the classic chain rule for the KL Divergence. In the general case, a classical requirement for FPD functionals~\cite[Theorem 4.1 Part 1]{karny2012axiomatisation} is violated, necessitating the novel construction to which this paper is devoted to. 
\end{remark}

Our main results will leverage the following lemma from~\cite{wang2021variational}.

\begin{lemma}\label{lem:densitysolution0}
Let $\bv{U}$ be a random variable with outcomes in a subset $\sU$ of $\R^{m}$. Let $p$  be the decision density function for $\bv{u}$ and $q$ be some given reference density. Let $J:\R^m\to\R$ be a proper (finite valued) convex lower semi-continuous function on the $m$-dimensional vector space of reals, constituting a cost function, and $\rho>0$ be a scaling parameter. Consider the following  optimization problem with respect to the density $p(\bv{u})$
    \begin{equation}\label{eq:gentsopt0}
\min_{p(\bv{u})} \frac{1}{\rho}\mathbb{E}_{p(\bv{U})} [J(\bv{U})] +\Dr{p(\bv{u})}{q(\bv{u})}.
    \end{equation}
The unique optimal solution to the problem, which always exists, is
    \begin{equation}\label{eq:gentssol0}
        p^\star(U):= \frac{\exp_r\left(-\bar{\rho}^{-1} J(\bv{U})\right)q(\bv{U})}{\int \exp_r\left(-\bar{\rho}^{-1} J\right)q(\bv{\bv{u}}) d\bv{u}},\,\,\text{with }\bar{\rho} = \lambda_0(r-1)\rho
    \end{equation}
    Here $\lambda_0$ is the Lagrange multiplier for the integrability constraint.
\end{lemma}
\begin{remark}
Lemma~\ref{lem:densitysolution0} shows that the optimal solution of \TsD~regularized optimization problem features a deformed exponential kernel applied to $q(\bv{u})$. This generalizes the classic result in entropy-regularized optimization where the optimal solution features an exponential kernel (e.g.~\cite[Lemma 1]{gagliardi2022probabilistic}).
\end{remark}

\subsection{A Primer on Fixed Point Theory}

We now recall elements from fixed point theory that are instrumental to prove our main results. A fixed point iteration (e.g.~\cite{zeidler1986nonlinear}) is a property of a sequential procedure of applying a map $M:\mathcal{X}\to\mathcal{X}$ from a space $\mathcal{X}$ to itself, i.e. $M(w_l)=w_{l+1}$ for a sequence $l\in\mathbb{N}$. This map is specifically a fixed point iteration when the sequence $\{w_l\}$ converges, with respect to an appropriate topology, to some element $w^\star\in\mathcal{X}$ of the space $\mathcal{X}$ such that $M(w^*)=w^*$., Formally, given a starting point $w_0$ chosen from (possibly some particular subset of) $\mathcal{X}$, the map $M$ can be applied iteratively so that:
\begin{equation}\label{eq:fixedptdef}
w_0\in\mathcal{C}\subseteq\mathcal{X},\,\,\,
w_{l+1}=M(w_l),\,\,\,
w_l\to w^\star,\,w^\star=M(w^\star),
\end{equation}
where $l$ is the iteration and convergence is typically taken to be in the strong (norm) topology of $\mathcal{X}$, although equivalent variants are valid in, e.g., weak or weak$^*$ topologies~\cite{buttazzo2014variational}.

The classic Banach/Caccioppoli fixed point theorem guarantees that a contraction within a metric space has a fixed point. This classic theorem can be found in, e.g.~\cite[Theorem 1.A]{zeidler1986nonlinear}
\begin{theorem}\label{th:banachfp}
Let $\mathcal{X}$ be a nonempty complete metric space with metric $d(\cdot,\cdot)$. If,
\begin{equation}\label{eq:contract}
    d\left(M(v),M(w)\right)\le \zeta d(\left(v,w\right)
\end{equation}
    with $0<\zeta<1$, then $M$ has a unique fixed point $w^\star$ and the iteration~\eqref{eq:fixedptdef} satisfies $w_{l}\to w^\star$ in the strong norm topology of $\mathcal{X}$.
\end{theorem}

We also recall Schauder's Fixed Point Theorem, which can be found in e.g.,~\cite[Theorem 2.A and Corollary 2.13]{zeidler1986nonlinear}
\begin{theorem}\label{th:schauder}
Let $\mathcal{X}$ be a Banach space. Let one of the following be satisfied:
\begin{enumerate}
    \item  $\mathcal{C}$ is a nonempty, closed, bounded and convex subset of $\mathcal{X}$ and  $M:\mathcal{C}\to \mathcal{C}$ is a compact operator;
    \item  $\mathcal{C}$ is a nonempty, compact,  and convex subset of $\mathcal{X}$ and  $M:\mathcal{C}\to \mathcal{C}$ is a continuous operator;
\end{enumerate}
then, $M$ has a fixed point $w^\star\in\mathcal{C}$. Moreover, if $w_0\in\mathcal{C}$, then for the sequence $w_l$  generated by  $w_{l+1}=M(w_l)$ it holds that $w_l\to w^\star$.\end{theorem}
A fixed point construction that satisfies either the assumptions of the Banach/Caccioppoli or Schauder fixed point theorems yields both a constructive proof for the existence of a fixed point and a procedure iteratively converging to it. 

\subsection{Fully Probabilistic Design}
To introduce the FPD problem (see references from Introduction) let: (i) $k=0,1,\cdots,N$ be a time-index; (ii) $\bv{X}_k\in\sX\subseteq\R^n$ be the environment/system state at time-step $k$; (iii) $\bv{U}_k\in\sU\subseteq\R^m$ be the action at time-step $k$.  The time indexing is chosen so that the environment transitions to $\bv{x}_k$ when $\bv{u}_k$ is applied.  The possibly non-stationary, nonlinear stochastic dynamics for the environment is $\plant{k}{k-1}$. At each $k$, the agent determines actions by sampling from a randomized policy, $\policy{k}{k-1}$; the closed-loop (agent-environment) behavior is captured by 
\begin{align}\label{eqn:behavior}
p_{0:N} := p_{0}\left(\bv{x}_0\right)\prod_{k=1}^N \jointxu{k}{k-1} = p_{0}\left(\bv{x}_0\right)\prod_{k=1}^N \plant{k}{k-1}\policy{k}{k-1},
\end{align}
where $p_{0}\left(\bv{x}_0\right)$ is a prior capturing initial conditions. Given this set-up,  FPD can be formalized with the following:
\begin{Problem}[Fully Probabilistic Design]\label{prob:FPD}
Let $p_{0:N}$ be defined as in~\eqref{eqn:behavior} and given a target $q_{0:N}$ defined as:
\begin{align}\label{eqn:refbehaviorb}
q_{0:N} := q_{0}\left(\bv{x}_0\right)\prod_{k=1}^N \refjointxu{k}{k-1} = q_{0}\left(\bv{x}_0\right)\prod_{k=1}^N \refplant{k}{k-1}\refpolicy{k}{k-1},
\end{align}
find $\left\{\optimalpolicy{k}{k-1}\right\}_{k\in[N]}$ such that: 
\begin{equation}\label{eqn:main_problemb}
    \begin{aligned}
{\left\{\optimalpolicy{k}{k-1}\right\}_{k\in[N]}}\in & \underset{\left\{ \policy{k}{k-1}\right\}_{k\in[N]}}{\mathrm{argmin}}
    \DKL{p_{0:N}}{q_{0:N}}  \\
   &  s.t.  \  \policy{k}{k-1}\in\sD \ \ \forall k\in [N].
    \end{aligned}
\end{equation}
\end{Problem}

This problem has a unique solution which has a well-defined closed form expression. In particular, the classical result in~\cite[Theorem 6.1]{karny2012axiomatisation} can be stated as follows:
\begin{theorem}\label{thm:klfpdsoln}
    Given $\{\plant{k}{k-1},\refpolicy{k}{k-1},\refplant{k}{k-1}\}$, the optimal ${\left\{\optimalpolicy{k}{k-1}\right\}_{k\in[N]}}$ for solving Problem~\ref{prob:FPD} can be computed using a backwards induction procedure.
    
    Specifically, sequentially for $k=N,N-1,\cdots,1$ the unique optimal decision $\sU^*:=\{\optimalpolicy{k}{k-1}\}$ can be obtained by performing the following two computations:
    \[
   \gamma_{k-1}(\bv{x}_{k-1}) =\int\left(\refpolicy{k}{k-1}\exp\left\{-
   \omega_k(\bv{u}_k,\bv{x}_{k-1})\right\}\right)d\bv{u}_k
    ,\, k<N,\,\gamma_N(\bv{x}_N)=1
    \]
    and
    \[
    \omega_k(\bv{u}_k,\bv{x}_{k-1})=\int p_k(\bv{u}_k,\bv{x}_{k-1})\log\left(\frac{\plant{k}{k-1}}{\gamma_k(\bv{x}_{k})\refplant{k}{k-1}}\right)d\bv{x}_k.
    \]
    The explicit expression for the closed form density becomes:
    \[
    \optimalpolicy{k}{k-1}= \refpolicy{k}{k-1}\frac{\exp\left\{-\omega_k(\bv{u}_k,\bv{x}_{k-1})\right\}}{\gamma_{k-1}(\bv{x}_{k-1})},
    \]
\end{theorem}



\begin{remark}
In Problem~\ref{prob:FPD}, $q_{0:N}$ is often termed in the FPD literature as ideal or reference probability and captures the {\em agent aim}~\cite{MK-SM:25}. This probability, which can be extracted from data or from first-principles~\cite{gagliardi2022probabilistic}, if also known as generative model in, e.g., the neuroscience literature; see, e.g.,~\cite{AS-HJ-KF-GR:25} and references therein.
\end{remark}

\section{Main results: Tsallis Fully Probabilistic Design}\label{sec:mainresults}

First, we formalize T-FPD. Then, we introduce a double loop iteration scheme to find the optimal solution. We then show that the iteration defines a contraction that provably converges to the optimal solution of the T-FPD problem. 

\subsection{Formulating T-FPD}

Tsallis Fully Probabilistic Design generalizes Problem~\ref{prob:FPD} by replacing the KL divergence with \TsD.

\begin{Problem}\label{prob:mainb}
Given the sequence of decision-state densities $p_{0:N}$ defined in~\eqref{eqn:behavior} together with reference densities  $q_{0:N}$ defined in~\eqref{eqn:refbehaviorb}, find $\left\{\optimalpolicy{k}{k-1}\right\}_{k\in[N]}$ such that: 
\begin{equation}\label{eqn:main_problemb}
    \begin{aligned}
{\left\{\optimalpolicy{k}{k-1}\right\}_{k\in[N]}}\in & \underset{\left\{ \policy{k}{k-1}\right\}_{k\in[N]}}{\mathrm{argmin}}
    \Dr{p_{0:N}}{q_{0:N}} \\
   &  s.t.  \  \policy{k}{k-1}\in\sD \ \ \forall k\in [N].
    \end{aligned}
\end{equation}
A priori, we do not assume a solution exists, and if it does, that it is unique, and so denote the, possibly empty and possibly singleton, set of solutions to~\eqref{eqn:main_problemb} as $\sU^*$.
\end{Problem}

The problem appears as the classical FPD, Problem~\ref{prob:FPD}, except that -- crucially -- the KL divergence term that constitutes the objective function to minimize is replaced by its Tsallis equivalent. Note that this problem is now parametrized by $r$, presenting a more flexible approach to FPD. Specifically, for $r>1$, as $r$ increases, the divergence becomes less sensitive to rare events, outliers, and tail behavior. A large value of $r$ becomes more appropriate when the distributions in the the reference and optimal densities are known to be highly skewed or with significant kurtosis, or come from a family associated with power law distributions, as in phenomena associated to financial markets, socioeconomics or climate. In addition, a large positive value of $r$ is appropriate for probabilistic decisions for systems with significant memory patterns across the stages $k$. Inversely, for $r<1$, the divergence becomes more sensitive to tail events relative to the KL divergence, and is appropriate for cases wherein the ground truth densities are expected to be symmetric while also repulsive relative to past behavior. 

\begin{remark}
Classic FPD proofs, including variants~\cite{gagliardi2022probabilistic} that embed moment constraints in the  formulation, rely on the use of the chain rule for the KL divergence. \TsD~is however non-additive and, as we shall see next, this means that a distinct mathematical approach is needed to find the optimal solution. 
\end{remark}


\subsection{Computing a Solution to T-FPD}

To streamline presentation, we give here an overview of the double iteration scheme that we then demonstrate provably converges to an optimal solution of Problem~\ref{prob:mainb}. The double iteration consists of an ``outer'' or ``major'' fixed point type iteration denoted by $l$, each of which consists of a complete backwards induction from $k=N$ to $k=1$. In Section~\ref{sec:expressions} we give the detailed expressions for the induction. As we shall see, as $r\to 1$ in Problem~\ref{prob:mainb}, the double iteration scheme approaches a procedure involving one complete FPD backwards induction defined as in Theorem~\ref{thm:klfpdsoln} followed by repeated applications of the identity map (see Remark~\ref{rem:rto1}).


We design a nested iteration to tackle Problem~\ref{prob:mainb}. In this case, we shall define a map depending on a step-size parameter $\alpha$, denoted as $T_{\alpha}:\mathcal{D}^N\to \mathcal{D}^N$, to perform the iterative procedure:
\begin{equation}\label{eq:tsdmapiter}
\left\{\optimalpolicygen{k}{k-1}{l+1}\right\}=T_{\alpha}\left(\left\{\optimalpolicygen{k}{k-1}{l}\right\}\right)
\end{equation}
This defines a sequence of solution estimates $\left\{\optimalpolicygen{k}{k-1}{l}\right\}$ and their associated state dynamics $\{p^{*,l}_k(\bv{x}_k)\}$ starting from some initial guess $\left\{\optimalpolicygen{k}{k-1}{0}\right\}$.

Each application of $T_{{\alpha}}$, in turn, applies a recursive backward induction $S_{l,k}$ wherein, for each $l=0,1,\cdots$, the iteration proceeds by,
\begin{equation}\label{eq:backindtsd}
\begin{array}{l}
\optimalpolicygen{N}{N-1}{l+}= S_{l,N}\left(\emptyset,\plant{N}{N-1},\refplant{N}{N-1},\refpolicy{N}{N-1},\left\{\optimalpolicygen{j}{j-1}{l}\right\}_{j\in[N]}\right) \\
\optimalpolicygen{k}{k-1}{l+} =S_{l,k}\left(\optimalpolicygen{k+1}{k}{l+} ,\plant{k}{k-1},\refplant{k}{k-1},\right.
\\ \qquad\qquad\qquad\qquad\qquad\left.\refpolicy{k}{k-1},\left\{\optimalpolicygen{j}{j-1}{l}\right\}_{j\in\{1,2,\cdots,k\}}\right)  ,\,k=N-1,N-2,\cdots, 1
\end{array}
\end{equation}
Observe now that the data for computing $S_{l,k}$ depends, throughout the backwards induction, on the previous major iteration solution $\left\{\optimalpolicygen{k}{k-1}{l}\right\}$. This can be thought of as a type of Gauss-Seidel method for solving an operator equation.

The evaluation of the map $T_{\alpha}$ performs a relaxation of the candidate defined by $l^+$ and the previous solution estimate $l$ with application of parameter $0<\alpha\le 1$:
\begin{equation}\label{eq:talpharelax}
\begin{array}{l}
    \left\{\optimalpolicygen{k}{k-1}{l+1}\right\} = T_{\alpha}\left(\left\{\optimalpolicygen{k}{k-1}{l}\right\}\right)\\ \qquad\qquad \qquad \quad :=\left\{ \alpha\optimalpolicygen{k}{k-1}{l+}+(1-\alpha)\optimalpolicygen{k}{k-1}{l},\,\forall k=1,2,\cdots,N\right\}
    \end{array}
\end{equation}
Observe that since this is a linear combination, it preserves the density normalization, i.e. 
\[
\optimalpolicygen{k}{k-1}{l}\in \sD \, \Leftrightarrow \left\|\optimalpolicygen{k}{k-1}{l}\right\|_{L^1(\R^m)} = 1,\,\forall l\in\mathbb{N},k\in[N],\bv{x}_{k-1}\in\R^n
\]



\subsection{Backwards Induction Expressions}\label{sec:expressions}

The specific form of the computation of $S_{l,k}$ will be constructed so as to utilize Lemma~\ref{lem:densitysolution0}. To ameliorate density of the equations through the derivations, the following shorthand notation shall be employed:
\[
\begin{array}{lll}
q^{(x)}_k:= \refplant{k}{k-1},&
q^{(u)}_k:= \refpolicy{k}{k-1}, & q_k:=q_k(\bv{x}_k) \\
p^{(x)}_k:=\plant{k}{k-1}, & p^{(u)}_k := \policy{k}{k-1} & p_k:=p_k(\bv{x}_k)
\end{array}
\]
\[
\begin{array}{l}
\bar{p}^l_{0:k}:= p_{0}\left(\bv{x}_0\right)\prod_{j=1}^{k} \plant{j}{j-1}\optimalpolicygen{j}{j-1}{l}  
\end{array}
\]
In the above expression, $\optimalpolicygen{j}{j-1}{l} $ is the current estimate for the T-FPD solution (i.e., the outcome of the last outer iteration). 

The constructions for $S_{l,k}$ are specifically designed in order to satisfy several requirements. First, as in the original FPD result, Theorem~\ref{thm:klfpdsoln}, they constitute the evaluation of closed form analytical expressions in a backwards induction operation from state $k=N$ to $k=1$. In order to mitigate the challenges presented by Tsallis divergence subadditivity, this property is facilitated by replacing certain expressions for $\optimalpolicy{k}{k-1}$ and $p(\bv{x}_k)$ with their counterparts at the previous outer iteration, $\optimalpolicygen{k}{k-1}{l}$ and effectively treating them as constants. 

The overall process, upon completion, proceeds towards the next outer iteration with the newly computed $\optimalpolicygen{k}{k-1}{l+1}$ treated as constant terms during the backwards induction performed during the outer iteration $l+1\to l+2$. In Section~\ref{sec:convergence} we prove that this scheme converges to the optimal solution of Problem~\ref{prob:mainb}.

To motivate why it is necessary, for tractability, to substitute certain terms $\policy{k}{k-1}$ and $p_k(\bv{x}_k)$ with a fixed (constant) estimate (which become $\optimalpolicygen{k}{k-1}{l}$ and $p^{*,l}_k(\bv{x}_k)$ in our double loop construction), consider, hypothetically, the problem:
\[
\min_{p(\bv{u}\vert \bv{x})\in\sD} \, \Dr{p(\bv{u},\bv{x})}{q(\bv{u},\bv{x})} 
\]
Observe that the specific application of Lemma~\ref{lem:densitysolution0} to incorporate Lemma~\ref{lem:non-additivity} cannot be applied, as sub-additivity presents a bilinear form. However, consider that one has an initial guess $\{p^{*,l}(\bv{u}), p^{*,l}(\bv{x})\}$ to be used in a double loop scheme. The choice to be made, then, is to approximate the last term as either:
\[
(1-r)\E_{p(\bv{u})}\left[\Dr{p(\bv{x}\vert \bv{u})}{q(\bv{u})}\right]\Dr{p(\bv{x})}{q(\bv{x})}
\approx (1-r)\E_{p(\bv{u})}\left[\Dr{p(\bv{x}\vert \bv{u})}{q(\bv{u})}\right]\Dr{p^{*,l}(\bv{x})}{q(\bv{x})}
\]
or
\[
(1-r)\E_{p(\bv{u})}\left[\Dr{p(\bv{x}\vert \bv{u})}{q(\bv{u})}\right]\Dr{p(\bv{x})}{q(\bv{x})}
\approx
(1-r)\E_{p^{*,l}(\bv{u})}\left[\Dr{p(\bv{x}\vert \bv{u})}{q(\bv{u})}\right]\Dr{p(\bv{x})}{q(\bv{x})}
\]
Either enables the application of Lemma~\ref{lem:densitysolution0}. 

Correspondingly, some terms $\optimalpolicy{k}{k-1}$ in the decision optimization problem formulated in Problem~\ref{prob:mainb} are retained as free variables to be solved for and some terms are replaced by $\optimalpolicygen{k}{k-1}{l}$ and treated as constants. Below, we present one such choice and display the corresponding backwards induction solution. Although which term is chosen may influence some properties of the procedure, an arbitrary choice among the two achieves a constructive proof of solution existence, which is the purpose of this work.

This reasoning informs the construction in the next two results defining the explicit expressions for, first, the base case $k=N$ and then the inductive step from $k+1$ to $k$, for the backwards recursion defining each major fixed point type iteration. The first Theorem presents a construction for the base case, that is the form of $S_{l,N}$ given $\{\optimalpolicygen{k}{k-1}{l}\}$, the reference solution $\{\refpolicy{k}{k-1},\refplant{k}{k-1}\}$ and final transition $\{\plant{N}{N-1}\}$.



\begin{theorem}\label{th:solnonefpN}
(\textbf{Base case}) Consider solving the decision Problem~\ref{prob:mainb} at stage $N$ with respect to $p(\bv{u}_N\vert\bv{x}_{N-1})$, i.e., computing 
\[
S_{l,N}\left(\emptyset,\plant{N}{N-1},\refplant{N}{N-1},\refpolicy{N}{N-1},\left\{\optimalpolicygen{j}{j-1}{l}\right\}_{j\in[N]}\right)
\]
There exists an operator $S_{l,N}$ that constitutes the solution of the problem structured in the form~\eqref{eq:gentsopt0} defined in Lemma~\ref{lem:densitysolution0} with a corresponding solution~\eqref{eq:gentssol0} given by:
    \begin{equation}\label{eq:backfpsolNthm}
    \begin{array}{l}
        p^{*,l+}_N(\bv{u}_N\vert\bv{x}_{N-1}):= \frac{\exp_r\left(-\bar{\rho}_N^{-1}J_N^0+J^1_N \right)q(\bv{u}_N\vert \bv{x}_{N-1})}{\int \exp_r\left(-\bar{\rho}_N^{-1}J_N^0+J_N^1 \right)q(\bv{u}_N\vert \bv{x}_{N-1}) d\bv{u}_N},\\
        \text{with } \bar{\rho}_N = \lambda_0(r-1)\rho
        \end{array}
    \end{equation}
    with,
\begin{equation}\label{eq:backfpsolNtermsthm}
\begin{array}{l}
J^0_N=\left((1-r)\Dr{\optimalpolicygen{N}{N-1}{l}}{q^{(u)}_N}\right)\Dr{p^{(x)}_N}{q^{(x)}_N}
\\ \qquad +(1-r)^2 \Dr{\optimalpolicygen{N}{N-1}{l}}{q^{(u)}_N}\Dr{\bar{p}^{l}_{0:(N-1)}}{q_{0:(N-1)}}\Dr{p^{(x)}_N}{q^{(x)}_N} 
\\ J^1_N= \Dr{p^{(x)}_N}{q^{(x)}_N}
 \\ 
\rho = \left(1+ (1-r)\Dr{\bar{p}^{l}_{0:(N-1)}}{q_{0:(N-1)}}\right)
\end{array}
\end{equation}
\end{theorem}

Next, we derive a construction for the backwards inductive step from $k+1$ to $k$ at major iteration $l$. That is, an expression for $\optimalpolicygen{k}{k-1}{l+}$ given 1) the previously computed $\optimalpolicygen{j}{j-1}{l+}$ for $j>k$ and $j\le N$, 2) a previous estimate $\{\optimalpolicygen{j}{j-1}{l}\}$ for $j\ge 1$ and $j\le k$, 3) the reference solution $\{\refpolicy{k}{k-1},\refplant{k}{k-1}\}$ and 4) current transition $\{\plant{k}{k-1}\}$. This expression constitutes the operator $S_{l,k}$.

\begin{theorem}\label{th:solnonefp} ({\bf Inductive expression})
There is a closed form expression for each iterative application of the operator \[
S_{l,k}\left(\optimalpolicygen{k+1}{k}{l+} ,\plant{k}{k-1},\refplant{k}{k-1},\refpolicy{k}{k-1},\left\{\optimalpolicygen{j}{j-1}{l}\right\}_{j\in[k]}\right)
\]
that computes the solution of a problem structured in the form~\eqref{eq:gentsopt0} defined in Lemma~\ref{lem:densitysolution0} with a corresponding solution of the form~\eqref{eq:gentssol0} for $\optimalpolicygen{k}{k-1}{l+}$ given by:
\begin{equation}\label{eq:backfpsolkthm}
    \begin{array}{l}
        \optimalpolicygen{k}{k-1}{l+}:= \frac{\exp_r\left(-\bar{\rho}_k^{-1}J_k \right)q_k(\bv{u}_k\vert \bv{x}_{k-1})}{\int \exp_r\left(-\bar{\rho}_k^{-1}J_k \right)q_k(\bv{u}_k\vert \bv{x}_{k-1}) d\bv{u}_k},\\
        \text{with } \bar{\rho}_k = \lambda_0(r-1)\rho_k
        \end{array}
    \end{equation}
    with $J_k$ and $\rho_k$ given by,
\begin{equation}\label{eq:backfpsolktermsthm}
\begin{array}{l}
J_k=\Dr{\shortplant{k}}{q^{(x)}_k}+(1-r) \Dr{\bar{p}^{l}_{0:(k-1)}}{q_{0:(k-1)}}\Dr{\shortplant{k}}{q^{(x)}_k} \\ 
\qquad+\E_{\shortplant{k}}\left[\Dr{p^{l+}_{(k+1):N}}{q_{(k+1):N}}\right]
 \\ \qquad +(1-r)\E_{p^l_{k-1}}\left[\Dr{\optimalpolicygen{k}{k-1}{l}}{q^{(u)}_k}\right]\Dr{\shortplant{k}}{q^{(x)}_k}\\ \qquad +
(1-r)^2\E_{\shortplant{k}\optimalpolicygen{k}{k-1}{l}}\left[\Dr{p^{l+}_{(k+1):N}}{q_{(k+1):N}}\right] \Dr{\shortplant{k}}{q^{(x)}_k} \\
\qquad +(1-r)^2\E_{\shortplant{k}\optimalpolicygen{k}{k-1}{l}}\left[\Dr{p^{l+}_{(k+1):N}}{q_{(k+1):N}}\right] \Dr{\bar{p}^{l}_{0:(k-1)}}{q_{0:(k-1)}}\\
\qquad +(1-r)^2 \E_{\shortplant{k}\optimalpolicygen{k}{k-1}{l}}\left[\Dr{p^{l+}_{(k+1):N}}{q_{(k+1):N}}\right] \\ \qquad\qquad\qquad\times\E_{p^l_{k-1}}\left[\Dr{\optimalpolicygen{k}{k-1}{l}}{q^{(u)}_k}\right]\Dr{\shortplant{k}}{q^{(x)}_k} \\
\qquad + (1-r)^3\E_{p^l_{k-1}\shortplant{k}\optimalpolicygen{k}{k-1}{l}}\left[\Dr{p^{l+}_{(k+1):N}}{q_{(k+1):N}}\right]\\ \qquad\qquad\qquad\times \Dr{\bar{p}^l_{0:(k-1)}}{q^l_{0:(k-1)}}\E_{p^l_{k-1}}\left[\Dr{\optimalpolicygen{k}{k-1}{l}}{q^{(u)}_k}\right]\Dr{\shortplant{k}}{q^{(x)}_k}
\\
\rho_k = 1+(1-r)\Dr{p_{0:(k-1)}}{q_{0:(k-1)}}\\
\qquad+\E_{p^l_{k-1}\shortplant{k}\optimalpolicygen{k}{k-1}{l}}\left[\Dr{p^{l+}_{(k+1):N}}{q_{(k+1):N}}\right]
\end{array}
\end{equation}
\end{theorem}

\begin{remark}\label{rem:rto1}
    Observe that in the expressions for both Theorem~\ref{th:solnonefpN} and~\ref{th:solnonefp}, as $r \to 1$, many of the terms, in particular those scaled by $(1-r)$, disappear. Without these terms, it can be seen that the expressions become equivalent to the original FPD with KL divergence, that is, as defined in Theorem~\ref{thm:klfpdsoln}.
\end{remark}

\subsection{Theoretical Properties}\label{sec:convergence}
\begin{remark} In an earlier version, the manuscript claimed theoretical results demonstrating that the map $T_{\alpha}$ incorporating the nested backwards induction $S_{l,k}$ defined above is a fixed point iteration and is, moreover, convergent to a solution of Problem~\ref{prob:mainb}. The claims were later found to be mistaken, and the existence of a solution to Problem~\ref{prob:mainb}, the fixed point guarantees of $T_{\alpha}$ and the convergence guarantees of the scheme are the topics of ongoing work.\end{remark}

\begin{lemma}\label{lem:tsdprop}
    The Tsallis Divergence $\Dr{p(\bv{u})}{q(\bv{u})}$ satisfies:
    \begin{enumerate}
        \item The quantity exhibits quadratic growth with respect to variations in the first argument, i.e., there exists a $q>0$ and $C_q\in\mathbb{R}$ such that for all $h(\bv{u})$ such that $p(\bv{u})+\epsilon h(\bv{u})\sD $ for sufficiently small $\epsilon>0$,
        \[
        \delta^2 \Dr{p(\bv{u})}{q(\bv{u})}[h(\bv{u}),h(\bv{u})]\ge c_q\left\|h(\bv{u})\right\|^2_{L^1(\R^d)} +C_q
        \]
        \item For all $p>0$, it holds that
\[
\|p(\bv{u})-q(\bv{u})\|^2_{L^1(\Re^d)}\le C(p) \Dr{p(\bv{u})}{q(\bv{u})} 
\]
This statement is called the \emph{Generalized Pinsker Inequality}. Observe that this implies that $\Dr{p(\bv{u})}{q(\bv{u})}$ is lower semicontinuous with respect to $p(\bv{u})$ in the strong norm topology of $L^1(\Re^d)$. 
    \end{enumerate}
\end{lemma}
\begin{proof}
    For the first statement, rewrite the Tsallis Divergence as:
\[
\Dr{p(\bv{u})}{q(\bv{u})}\equiv \frac{1}{r-1}\left(\int p(\bv{u})^r q(\bv{u})^{1-r} d\bv{u}-1\right)
\]
Now take the first, then second variation of $\Dr{p(\bv{u})}{q(\bv{u})}$ by computing
\[
\frac{d^2\Dr{p(\bv{u})+\epsilon h(\bv{u}) }{q(\bv{u})}}{d\epsilon^2}\Bigg\vert_{\epsilon=0}
\]
with perturbation $h(\bv{u})\in\sD$. We obtain:
\[
\delta \Dr{p(\bv{u})}{q(\bv{u})}[h(\bv{u})] = \frac{r}{r-1}\int h(\bv{u}) p(\bv{u})^{r-1} q(\bv{u})^{1-r} d\bv{u}
\]
and
\[
\delta^2 \Dr{p(\bv{u})}{q(\bv{u})}[h(\bv{u}),h(\bv{u})] = r \int h(\bv{u})^2 p(\bv{u})^{r-2} q(\bv{u})^{1-r} d\bv{u}
\]
We can observe quadratic growth with respect to any $h(\bv{u})$ for all $p(\bv{v})$ and $q(\bv{u})$. 

The second statement is established in the literature for different values of $r$ and levels of bound tightness across~\cite{gilardoni2010pinsker,reid2009generalised,beretta2026generalized}.
\end{proof}

Observe that a fixed point of $T_{\alpha}$, i.e., a set of densities satisfying,
    \[
    \left\{\optimalpolicygen{k}{k-1}{*}\right\} = T_{\alpha}\left(\left\{\optimalpolicygen{k}{k-1}{*}\right\}\right)
    \]
    solves Problem~\ref{prob:mainb}.

\section{Derivation of the Solution Maps}\label{sec:proofs}

Begin with the base case, that is $k=N$.
\begin{proof} \emph{of Theorem}~\ref{th:solnonefpN}.
Consider, to start with, writing the objective function appearing in~\eqref{eqn:main_problemb} and applying sub-additivity~\eqref{eq:non-additive}:
\begin{equation}\label{eq:sumndecompose}
\begin{array}{l}
\Dr{p_{0:N}}{q_{0:N}} = \E_{p_{N-1}}\left[\Dr{\shortplant{N}\shortpolicy{N}{N-1}}{q_N}\right]
\\ \quad +\Dr{p_{0:(N-1)}}{q_{0:(N-1)}}+(1-r)\E_{p_{N-1}}\left[\Dr{\shortplant{N}\shortpolicy{N}{N-1}}{q_N}\right]\Dr{p_{0:(N-1)}}{q_{0:(N-1)}} 
\end{array}
\end{equation}
Now, collect the components of the sum that only include the last time period in the FPD problem, and for the extra term arising from subadditivity, replace the generic density terms $\{\policy{k}{k-1}\}_{k\in[N-1]}$ with the estimates from the previous fixed point iteration $l$, that is, 
the set $\left\{\optimalpolicygen{k}{k-1}{l}\right\}_{k\in[N-1]}$ to obtain:
\[
\begin{array}{l}
\E_{p_{N-1}}\left[\Dr{\shortplant{N}\shortpolicy{N}{N-1}}{q_N}\right]
\\ \quad +(1-r)\E_{p_{N-1}}\left[\Dr{\shortplant{N}\shortpolicy{N}{N-1}}{q_N}\right]\Dr{\bar{p}^{l}_{0:(N-1)}}{q_{0:(N-1)}} 
\end{array}
\]

Subadditivity as in Lemma~\ref{lem:non-additivity} applied to the Tsallis divergence term yields:
\[
\begin{array}{l}
\Dr{\shortplant{N}\shortpolicy{N}{N-1}}{q_N} = \E_{\shortpolicy{N}{N-1}}\left[\Dr{\shortplant{N}}{q^{(x)}_N}\right]+ \Dr{\shortpolicy{N}{N-1}}{q^{(u)}_N} \\
\qquad + (1-r) \E_{\shortpolicy{N}{N-1}}\left[\Dr{\shortplant{N}}{q^{(x)}_N}\right]\Dr{\shortpolicy{N}{N-1}}{q^{(u)}_N}
\end{array}
\]
The last term is now bilinear with respect to $\shortpolicy{N}{N-1}$, preventing direct application of Lemma~\ref{lem:densitysolution0}. 
It is at this point that we constructively define $S_{l,N}$ to be both tractable and convergent, by selectively fixing one of these $\shortpolicy{N}{N-1}$ terms to $\optimalpolicygen{N}{N-1}{l}$. We choose to replace the regularization, $\Dr{\shortpolicy{N}{N-1}}{q^{(u)}_N}$ to instead appear as $\Dr{\optimalpolicygen{N}{N-1}{l}}{q^{(u)}_N}$. The final form of the complete expression:
\[
\begin{array}{l}
\E_{p_{N-1}}\left[\E_{\shortpolicy{N}{N-1}}\left[\Dr{\shortplant{N}}{q^{(x)}_N}\right]+ \Dr{\shortpolicy{N}{N-1}}{q^{(u)}_N} \right]
\\ \quad +(1-r)\E_{p_{N-1}}\left[ \E_{\shortpolicy{N}{N-1}}\left[\Dr{\shortplant{N}}{q^{(x)}_N}\right]\Dr{\optimalpolicygen{N}{N-1}{l}}{q^{(u)}_N}\right]
\\ \quad +(1-r)\E_{p_{N-1}}\left[\E_{\shortpolicy{N}{N-1}}\left[\Dr{\shortplant{N}}{q^{(x)}_N}\right]\right]\Dr{\bar{p}^{l}_{0:(N-1)}}{q_{0:(N-1)}} 
\\ \quad +(1-r)\E_{p_{N-1}}\left[\Dr{\shortpolicy{N}{N-1}}{q^{(u)}_N}\right]\Dr{\bar{p}^{l}_{0:(N-1)}}{q_{0:(N-1)}} 
\\ \quad +(1-r)^2\E_{p_{N-1}}\left[ \E_{\shortpolicy{N}{N-1}}\left[\Dr{\shortplant{N}}{q^{(x)}_N}\right]\Dr{\optimalpolicygen{N}{N-1}{l}}{q^{(u)}_N}\right]\Dr{\bar{p}^{l}_{0:(N-1)}}{q_{0:(N-1)}} 
\end{array}
\]
Observe that, as claimed, this expression is of the form~\eqref{eq:gentsopt0} defined in Lemma~\ref{lem:densitysolution0}

Now, condition on $x_{N-1}$ and obtain the conditional minimum of this expression with respect to $\shortpolicy{N}{N-1}$. We observe that $\mathbb{E}_{p_{N-1}}$ is applied to the entire expression, and thus we can obtain a minimum by obtaining a minimum with respect to $\shortpolicy{N}{N-1}$ for every possible value of $x_{N-1}$. To this expression, we are ready to apply Lemma~\ref{lem:densitysolution0} with 
\[
\begin{array}{l}
q(U) = q^{(u)}_N \\
J =  
\left((1-r)\Dr{\optimalpolicygen{N}{N-1}{l}}{q^{(u)}_N}\right)\Dr{\shortplant{N}}{q^{(x)}_N} 
\\ \qquad +(1-r)^2 \Dr{\optimalpolicygen{N}{N-1}{l}}{q^{(u)}_N}\Dr{\bar{p}^{l}_{0:(N-1)}}{q_{0:(N-1)}}\Dr{\shortplant{N}}{q^{(x)}_N} 
\\ \qquad\qquad\qquad \qquad + \left(1+(1-r)\Dr{\bar{p}^{l}_{0:(N-1)}}{q_{0:(N-1)}}\right)\Dr{\shortplant{N}}{q^{(x)}_N}
 \\ \qquad := J_N^0 \\ \qquad\qquad\qquad \qquad + \rho J_N^1 \\ 
\rho = \left(1+ (1-r)\Dr{\bar{p}^{l}_{0:(N-1)}}{q_{0:(N-1)}}\right)
\end{array}
\]
To obtain the first backwards induction expression at time $N$ given the previous solution $l$:
    \begin{equation}\label{eq:backfpsolN}
    \begin{array}{l}
        \hat{p}^{*,l+1}_N(\bv{u}_N\vert\bv{x}_{N-1}):= \frac{\exp_r\left(-\bar{\rho}_N^{-1}J_N^0+J^1_N \right)q(\bv{u}_N\vert \bv{x}_{N-1})}{\int \exp_r\left(-\bar{\rho}_N^{-1}J_N^0+J_N^1 \right)q(\bv{u}_N\vert \bv{x}_{N-1}) d\bv{u}_N},\\
        \text{with } \bar{\rho}_N = \lambda_0(r-1)\rho
        \end{array}
    \end{equation}

\end{proof}

We can now prove Theorem~\ref{th:solnonefp}.

\begin{proof}\emph{ of Theorem}~\ref{th:solnonefp}.

Collect the components of the full sum, as in~\eqref{eq:sumndecompose}, and cleave and pivot the expression with
respect to stage $k$:
\[
\begin{array}{l}
\Dr{p_{0:N}}{q_{0:N}} = 
\Dr{p_{0:k}}{q_{0:k}} 
\\ \qquad +\E_{p_k}\left[\Dr{p_{(k+1):N}}{q_{(k+1):N}}\right] + (1-r) 
\Dr{p_{0:k}}{q_{0:k}} \E_{p_k}\left[\Dr{p_{(k+1):N}}{q_{(k+1):N}} \right] \\ 
= \Dr{p_{0:(k-1)}}{q_{0:(k-1)}} + \E_{p_{k-1}}\left[\Dr{\shortplant{k}\shortpolicy{k}{k-1}}{q_k}\right] \\ 
\qquad + (1-r) \Dr{p_{0:(k-1)}}{q_{0:(k-1)}}\E_{p_{k-1}}\left[\Dr{\shortplant{k}\shortpolicy{k}{k-1}}{q_k} \right]+ \E_{p_k}\left[\Dr{p_{(k+1):N}}{q_{(k+1):N}}\right] \\ 
\qquad + (1-r)  \E_{p_k}\left[\Dr{p_{(k+1):N}}{q_{(k+1):N}}\right]
\left[ \Dr{p_{0:(k-1)}}{q_{0:(k-1)}} + \E_{p_{k-1}}\left[\Dr{\shortplant{k}\shortpolicy{k}{k-1}}{q_k}\right]\right.
\\ \qquad \left.+ (1-r) \Dr{p_{0:(k-1)}}{q_{0:(k-1)}}\E_{p_{k-1}}\left[\Dr{\shortplant{k}\shortpolicy{k}{k-1}}{q_k}\right]\right] 
\end{array}
\]

Noting that the intention is to obtain an optimization problem with respect to $\optimalpolicy{k}{k-1}$, discard the terms
in the expression that are independent of this quantity and rearrange the expression:
\[
\begin{array}{l}
\E_{p_{k-1}}\left[\Dr{\shortplant{k}\shortpolicy{k}{k-1}}{q_k}\right]
+ (1-r) \Dr{p_{0:(k-1)}}{q_{0:(k-1)}}\E_{p_{k-1}}\left[\Dr{\shortplant{k}\shortpolicy{k}{k-1}}{q_k}\right] \\ \qquad + \E_{p_{k-1}\shortplant{k}\shortpolicy{k}{k-1}}\left[\Dr{p_{(k+1):N}}{q_{(k+1):N}}\right] \\ 
\qquad + (1-r)  \E_{p_{k-1}\shortplant{k}\shortpolicy{k}{k-1}}\left[\Dr{p_{(k+1):N}}{q_{(k+1):N}}\right]  \Dr{p_{0:(k-1)}}{q_{0:(k-1)}}\\ 
\qquad + (1-r)  \E_{p_{k-1}\shortplant{k}\shortpolicy{k}{k-1}}\left[\Dr{p_{(k+1):N}}{q_{(k+1):N}}\right]
\E_{p_{k-1}}\left[\Dr{\shortplant{k}\shortpolicy{k}{k-1}}{q_k}\right] \\ \qquad + (1-r)^2  \E_{p_{k-1}\shortplant{k}\shortpolicy{k}{k-1}}\left[\Dr{p_{(k+1):N}}{q_{(k+1):N}}\right] \Dr{p_{0:(k-1)}}{q_{0:(k-1)}}\E_{p_{k-1}}\left[\Dr{\shortplant{k}\shortpolicy{k}{k-1}}{q_k}\right] 
\end{array}
\]
Apply the chain rule again to split terms involving $\Dr{\shortplant{k}\shortpolicy{k}{k-1}}{q_k}$ to obtain:

\[
\begin{array}{l}
\E_{p_{k-1}}\left[\Dr{\shortpolicy{k}{k-1}}{q^{(u)}_k}\right] + \E_{p_{k-1}\shortpolicy{k}{k-1}}\left[\Dr{\shortplant{k}}{q^{(x)}_k}\right] \\ 
\qquad 
 + (1-r) \E_{p_{k-1}}\left[\Dr{\shortpolicy{k}{k-1}}{q^{(u)}_k}\right]\E_{p_{k-1}\shortpolicy{k}{k-1}}\left[\Dr{\shortplant{k}}{q^{(x)}_k}\right]
\\ 
\qquad
+ (1-r) \Dr{p_{0:(k-1)}}{q_{0:(k-1)}}\E_{p_{k-1}}\left[\Dr{\shortpolicy{k}{k-1}}{q^{(u)}_k}\right]
\\ 
\qquad
+ (1-r) \Dr{p_{0:(k-1)}}{q_{0:(k-1)}}\E_{p_{k-1}\shortpolicy{k}{k-1}}\left[\Dr{\shortplant{k}}{q^{(x)}_k}\right]
\\ 
\qquad
+ (1-r)^2 \Dr{p_{0:(k-1)}}{q_{0:(k-1)}}\left[\Dr{\shortpolicy{k}{k-1}}{q^{(u)}_k}\right]\E_{p_{k-1}\shortpolicy{k}{k-1}}\left[\Dr{\shortplant{k}}{q^{(x)}_k}\right]
\\ \qquad + \E_{p_{k-1}\shortplant{k}\shortpolicy{k}{k-1}}\left[\Dr{p_{(k+1):N}}{q_{(k+1):N}}\right] \\ 
\qquad + (1-r)  \E_{p_{k-1}\shortplant{k}\shortpolicy{k}{k-1}}\left[\Dr{p_{(k+1):N}}{q_{(k+1):N}}\right]  \Dr{p_{0:(k-1)}}{q_{0:(k-1)}}\\ 
\qquad + (1-r)  \E_{p_{k-1}\shortplant{k}\shortpolicy{k}{k-1}}\left[\Dr{p_{(k+1):N}}{q_{(k+1):N}}\right]
\E_{p_{k-1}}\left[\Dr{\shortplant{k}\shortpolicy{k}{k-1}}{q_k}\right] \\ \qquad + (1-r)^2  \E_{p_{k-1}\shortplant{k}\shortpolicy{k}{k-1}}\left[\Dr{p_{(k+1):N}}{q_{(k+1):N}}\right]  \Dr{p_{0:(k-1)}}{q_{0:(k-1)}}\E_{p_{k-1}}\left[\Dr{\shortplant{k}\shortpolicy{k}{k-1}}{q_k}\right]  
\end{array}
\]
Applying sub-additivity to the terms $\E_{p_{k-1}}\left[\Dr{\shortplant{k}\shortpolicy{k}{k-1}}{q_k}\right]$, we obtain the complete expression:

\begin{equation}\label{eq:generalexpandk}
\begin{array}{l}
\E_{p_{k-1}}\left[\Dr{\shortpolicy{k}{k-1}}{q^{(u)}_k}\right] + \E_{p_{k-1}\shortpolicy{k}{k-1}}\left[\Dr{\shortplant{k}}{q^{(x)}_k}\right] \\ 
\qquad 
 + (1-r) \E_{p_{k-1}}\left[\Dr{\shortpolicy{k}{k-1}}{q^{(u)}_k}\right]\E_{p_{k-1}\shortpolicy{k}{k-1}}\left[\Dr{\shortplant{k}}{q^{(x)}_k}\right]
\\ 
\qquad
+ (1-r) \Dr{p_{0:(k-1)}}{q_{0:(k-1)}}\E_{p_{k-1}}\left[\Dr{\shortpolicy{k}{k-1}}{q^{(u)}_k}\right]
\\ 
\qquad
+ (1-r) \Dr{p_{0:(k-1)}}{q_{0:(k-1)}}\E_{p_{k-1}\shortpolicy{k}{k-1}}\left[\Dr{\shortplant{k}}{q^{(x)}_k}\right]
\\ 
\qquad
+ (1-r)^2 \Dr{p_{0:(k-1)}}{q_{0:(k-1)}}\left[\Dr{\shortpolicy{k}{k-1}}{q^{(u)}_k}\right]\E_{p_{k-1}\shortpolicy{k}{k-1}}\left[\Dr{\shortplant{k}}{q^{(x)}_k}\right]
\\ \qquad + \E_{p_{k-1}\shortplant{k}\shortpolicy{k}{k-1}}\left[\Dr{p_{(k+1):N}}{q_{(k+1):N}}\right] \\ 
\qquad + (1-r)  \E_{p_{k-1}\shortplant{k}\shortpolicy{k}{k-1}}\left[\Dr{p_{(k+1):N}}{q_{(k+1):N}}\right]  \Dr{p_{0:(k-1)}}{q_{0:(k-1)}}\\ 
\qquad + (1-r)  \E_{p_{k-1}\shortplant{k}\shortpolicy{k}{k-1}}\left[\Dr{p_{(k+1):N}}{q_{(k+1):N}}\right]
\E_{p_{k-1}}\left[\Dr{\shortpolicy{k}{k-1}}{q^{(u)}_k}\right]\\ 
\qquad + (1-r)  \E_{p_{k-1}\shortplant{k}\shortpolicy{k}{k-1}}\left[\Dr{p_{(k+1):N}}{q_{(k+1):N}}\right]
\E_{p_{k-1}\shortpolicy{k}{k-1}}\left[\Dr{\shortplant{k}}{q^{(x)}_k}\right] \\ 
\qquad + (1-r)^2  \E_{p_{k-1}\shortplant{k}\shortpolicy{k}{k-1}}\left[\Dr{p_{(k+1):N}}{q_{(k+1):N}}\right] \\ 
\qquad\qquad\qquad \times
\E_{p_{k-1}}\left[\Dr{\shortpolicy{k}{k-1}}{q^{(u)}_k}\right]\E_{p_{k-1}\shortpolicy{k}{k-1}}\left[\Dr{\shortplant{k}}{q^{(x)}_k}\right]
\\ \qquad + (1-r)^2  \E_{p_{k-1}\shortplant{k}\shortpolicy{k}{k-1}}\left[\Dr{p_{(k+1):N}}{q_{(k+1):N}}\right] \\ 
\qquad\qquad\qquad \times \Dr{p_{0:(k-1)}}{q_{0:(k-1)}}\E_{p_{k-1}}\left[\Dr{\shortpolicy{k}{k-1}}{q^{(u)}_k}\right]
\\ \qquad + (1-r)^2  \E_{p_{k-1}\shortplant{k}\shortpolicy{k}{k-1}}\left[\Dr{p_{(k+1):N}}{q_{(k+1):N}}\right] \\ 
\qquad\qquad\qquad \times \Dr{p_{0:(k-1)}}{q_{0:(k-1)}}\E_{p_{k-1}\shortpolicy{k}{k-1}}\left[\Dr{\shortplant{k}}{q^{(x)}_k}\right] 
\\ \qquad + (1-r)^3  \E_{p_{k-1}\shortplant{k}\shortpolicy{k}{k-1}}\left[\Dr{p_{(k+1):N}}{q_{(k+1):N}}\right] \\ 
\qquad\qquad\qquad \times \Dr{p_{0:(k-1)}}{q_{0:(k-1)}}\E_{p_{k-1}}\left[\Dr{\shortpolicy{k}{k-1}}{q^{(u)}_k}\right]\E_{p_{k-1}\shortpolicy{k}{k-1}}\left[\Dr{\shortplant{k}}{q^{(x)}_k}\right]
\end{array}
\end{equation}
Finally, we must perform appropriate substitutions, first with respect to the variable terms $\optimalpolicy{j}{j-1}$ for $j<k$ from $\{\optimalpolicygen{j}{j-1}{l}\}$
and for $j>k$ from $\{\optimalpolicygen{j}{j-1}{l+}\}$. The following substitutions will be chosen for the bilinear terms with respect to $\optimalpolicy{k}{k-1}$:
\[
\begin{array}{l}
\E_{p_{k-1}}\left[\Dr{\shortpolicy{k}{k-1}}{q^{(u)}_k}\right]\E_{p_{k-1}\shortpolicy{k}{k-1}}\left[\Dr{\shortplant{k}}{q^{(x)}_k}\right] \\ \qquad \Rightarrow
\E_{p^l_{k-1}}\left[\Dr{\optimalpolicygen{k}{k-1}{l}}{q^{(u)}_k}\right]\E_{p_{k-1}\shortpolicy{k}{k-1}}\left[\Dr{\shortplant{k}}{q^{(x)}_k}\right]
\\ 
\E_{p_{k-1}\shortplant{k}\shortpolicy{k}{k-1}}\left[\Dr{p^{l+}_{(k+1):N}}{q_{(k+1):N}}\right]
\E_{p_{k-1}}\left[\Dr{\shortpolicy{k}{k-1}}{q^{(u)}_k}\right] \\ \qquad 
\Rightarrow 
\E_{p_{k-1}\shortplant{k}\optimalpolicygen{k}{k-1}{l}}\left[\Dr{p^{l+}_{(k+1):N}}{q_{(k+1):N}}\right]
\E_{p_{k-1}}\left[\Dr{\shortpolicy{k}{k-1}}{q^{(u)}_k}\right]
\\ 
 \E_{p_{k-1}\shortplant{k}\shortpolicy{k}{k-1}}\left[\Dr{p^{l+}_{(k+1):N}}{q_{(k+1):N}}\right]
\E_{p_{k-1}\shortpolicy{k}{k-1}}\left[\Dr{\shortplant{k}}{q^{(x)}_k}\right] \\ \qquad  \Rightarrow
\E_{p_{k-1}\shortplant{k}\optimalpolicygen{k}{k-1}{l}}\left[\Dr{p^{l+}_{(k+1):N}}{q_{(k+1):N}}\right]
\E_{p_{k-1}\shortpolicy{k}{k-1}}\left[\Dr{\shortplant{k}}{q^{(x)}_k}\right]
 \\ 
 \E_{p_{k-1}\shortplant{k}\shortpolicy{k}{k-1}}\left[\Dr{p_{(k+1):N}}{q_{(k+1):N}}\right] \Dr{p_{0:(k-1)}}{q_{0:(k-1)}}\E_{p_{k-1}}\left[\Dr{\shortpolicy{k}{k-1}}{q^{(u)}_k}\right] \\ \qquad \Rightarrow
  \E_{p^l_{k-1}\shortplant{k}\optimalpolicygen{k}{k-1}{l}}\left[\Dr{p_{(k+1):N}}{q_{(k+1):N}}\right] \Dr{p_{0:(k-1)}}{q_{0:(k-1)}}\E_{p_{k-1}}\left[\Dr{\shortpolicy{k}{k-1}}{q^{(u)}_k}\right] 
\\ \E_{p_{k-1}\shortplant{k}\shortpolicy{k}{k-1}}\left[\Dr{p_{(k+1):N}}{q_{(k+1):N}}\right] \Dr{p_{0:(k-1)}}{q_{0:(k-1)}}\E_{p_{k-1}\shortpolicy{k}{k-1}}\left[\Dr{\shortplant{k}}{q^{(x)}_k}\right] \\ \qquad \Rightarrow
  \E_{p^l_{k-1}\shortplant{k}\optimalpolicygen{k}{k-1}{l}}\left[\Dr{p_{(k+1):N}}{q_{(k+1):N}}\right] \Dr{p_{0:(k-1)}}{q_{0:(k-1)}}\E_{p_{k-1}}\left[\Dr{\shortpolicy{k}{k-1}}{q^{(u)}_k}\right] 
\end{array}
\]
Perform the following substitutions for the trilinear terms:
\[
\begin{array}{l}
\E_{p_{k-1}\shortplant{k}\shortpolicy{k}{k-1}}\left[\Dr{p_{(k+1):N}}{q_{(k+1):N}}\right] \\ 
\qquad\qquad\qquad \times
\E_{p_{k-1}}\left[\Dr{\shortpolicy{k}{k-1}}{q^{(u)}_k}\right]\E_{p_{k-1}\shortpolicy{k}{k-1}}\left[\Dr{\shortplant{k}}{q^{(x)}_k}\right]
\\
\qquad \Rightarrow 
\E_{p^l_{k-1}\shortplant{k}\optimalpolicygen{k}{k-1}{l}}\left[\Dr{p^{l+}_{(k+1):N}}{q_{(k+1):N}}\right] \\ 
\qquad\qquad\qquad \times
\E_{p^l_{k-1}}\left[\Dr{\optimalpolicygen{k}{k-1}{l}}{q^{(u)}_k}\right]\E_{p_{k-1}\shortpolicy{k}{k-1}}\left[\Dr{\shortplant{k}}{q^{(x)}_k}\right]
\\ 
 \E_{p_{k-1}\shortplant{k}\shortpolicy{k}{k-1}}\left[\Dr{p_{(k+1):N}}{q_{(k+1):N}}\right] \\ 
\qquad\qquad\qquad \times \Dr{p_{0:(k-1)}}{q_{0:(k-1)}}\E_{p_{k-1}}\left[\Dr{\shortpolicy{k}{k-1}}{q^{(u)}_k}\right]\E_{p_{k-1}\shortpolicy{k}{k-1}}\left[\Dr{\shortplant{k}}{q^{(x)}_k}\right] 
\\
\qquad \Rightarrow 
 \E_{p^l_{k-1}\shortplant{k}\optimalpolicygen{k}{k-1}{l}}\left[\Dr{p^{l+}_{(k+1):N}}{q_{(k+1):N}}\right] \\ 
\qquad\qquad\qquad \times \Dr{p^l_{0:(k-1)}}{q^l_{0:(k-1)}}\E_{p^l_{k-1}}\left[\Dr{\optimalpolicygen{k}{k-1}{l}}{q^{(u)}_k}\right]\E_{p^l_{k-1}\shortpolicy{k}{k-1}}\left[\Dr{\shortplant{k}}{q^{(x)}_k}\right] 
\end{array}
\]
With the two aforementioned substitutions, the expression~\eqref{eq:generalexpandk} is of the form given in Lemma~\ref{lem:densitysolution0}, from
which the conclusion of the Theorem can be drawn.

\end{proof}




\bibliographystyle{cas-model2-names}

\bibliography{refs.bib}

\end{document}